\documentclass[10pt]{bmc_article}
\usepackage{cite} 
\usepackage{url}  
\usepackage{ifthen}  
\usepackage{multicol}   
\usepackage[utf8]{inputenc} 
\usepackage{mathptmx}
\usepackage{latexsym,amssymb,amsmath,amsthm}
\usepackage{algorithm} 
\usepackage{algorithmic}
\usepackage{textcomp,graphicx,epsfig}

\usepackage[authoryear,comma,longnamesfirst,sectionbib]{natbib} 

\newtheorem{defi}{Definition}[section]

\newtheorem{coro}[defi]{Corollary}
\newtheorem{theo}[defi]{Theorem}

\newboolean{publ}

\newenvironment{bmcformat}{\baselineskip20pt\sloppy\setboolean{publ}{false}}{\baselineskip20pt\sloppy}

\begin{document}
\begin{bmcformat}

\title{Weighted Kolmogorov Smirnov testing: an alternative for Gene
  Set Enrichment Analysis}

\author{Konstantina Charmpi$^{1,2,3}$%
         \email{Konstantina Charmpi - Konstantina.Charmpi@imag.fr}
       \and 
         Bernard Ycart\correspondingauthor$^{1,2,3}$%
         \email{Bernard Ycart\correspondingauthor - Bernard.Ycart@imag.fr}%
       }


\address{%
\iid(1) Universit\'e Grenoble Alpes, France\\
\iid(2) Laboratoire Jean Kuntzmann, CNRS UMR5224, Grenoble, France\\
\iid(3) Laboratoire d'Excellence TOUCAN, France%
}%

\maketitle

\begin{abstract}
Gene Set Enrichment Analysis (GSEA) is a basic tool for genomic data
treatment. From a statistical point of view, the centering of its
test statistic does not 
allow the derivation of asymptotic results.
A test statistic with a different centering is proposed.
Under the null hypothesis, the convergence in distribution of 
the new test statistic is proved, using the theory of empirical
processes. The limiting distribution can be computed by Monte-Carlo
simulation. The test defined in this way has been called Weighted
Kolmogorov Smirnov (WKS) test. The fact that the evaluation of the
asymptotic distribution serves for many different gene sets
results in shorter computing times. Using
expression data from the GEO repository, tested against the MSig
Database C2, a comparison between the classical GSEA test and the new
procedure has been conducted. Our conclusion is that, beyond its
mathematical and algorithmic advantages, the WKS test could be more
informative in many cases, than the classical GSEA test.


\end{abstract}

\section{Introduction}
Since its definition by \cite{Subramanianetal05},
Gene Set Enrichment Analysis (GSEA) has been very successful, and it
may now be considered as 
the most basic tool of genomic data treatment: see
\cite{Bild05,Nametal08,Huangetal09} for reviews.
GSEA aims at
comparing a vector of numeric data indexed by the set of all genes, to
the genes contained in a given smaller gene set. 
The numeric data are typically obtained from a microarray
experiment. They may consist in expression levels, p-values,  
correlations, fold-changes, t-statistics, signal-to-noise ratios, etc.
The number associated to any given gene will be referred to as its
\emph{weight}. Many examples of such data
can be downloaded from the Gene Expression Omnibus (GEO) repository
(\cite{Edgar02}). The gene set may contain 
genes known to be associated to a given biological process, 
a cellular component, a type of cancer, etc. 
Thematic lists of such gene sets are given in the Molecular Signature (MSig)
database (\cite{Subramanianetal05}). The question to be answered is:
are the weights inside the gene set significantly high or
low, compared to weights in a random gene set of the same size?

Denote by $N$ the
total number of genes ($N\simeq 20\,000$ for the human genome). 
It will be convenient to identify the genes to $N$
regularly spaced points 
on the interval $[0,1]$, and their weights to the values of a
positive valued function $g$, defined on $[0,1]$: gene number $i$
corresponds to point $i/N$, and its weight $w_i$ to $g(i/N)$. In
\cite{Subramanianetal05}, the numbering of the genes is chosen so that weights are ranked 
in decreasing order. Thus, the weights usually appear to vary smoothly between
contiguous genes, and the function $g$ can be assumed to be continuous.

The gene set is included in 
the set of all genes. Let $n$ be its size.
In practice, $n$ ranges from a few tens to a few hundreds: $n$ is 
much smaller than $N$. With the identification above, it is considered
as a subset of size $n$ of the interval $[0,1]$, say $\{U_1,\ldots,U_n\}$. 
If there is no particular relation
between the weights and the gene set (null hypothesis), then the gene
set must be considered as a uniform random
sample without replacement of the set of all genes. The fact that the
gene set size $n$ is much smaller than $N$, justifies identifying the
distribution of 
a uniform $n$-sample without replacement of $\{1/N,\ldots,N/N\}$ to
that of a $n$-sample of points, 
uniformly distributed on $[0,1]$. Therefore, the
null hypothesis is:  
\begin{center}
H0: The gene set is a $n$-tuple $(U_1,\ldots,U_n)$ of independent,
identically distributed (i.i.d.) random variables, uniformly distributed on the
interval $[0,1]$.  
\end{center}
The basic object is the following step function, cumulating the
proportion of weights
inside the gene set, along the interval $[0,1]$. It is defined
for all $t$ between $0$ and $1$ by:
\begin{equation}
\label{ECSP}
S_n(t) = \frac{\sum_{k=1}^n g(U_k)\,\mathbb{I}_{U_k \leqslant t}}
{\sum_{k=1}^n g(U_k)}
\;,
\end{equation}
where $\mathbb{I}$ denotes the indicator of an event.
The test statistic proposed by \cite{Subramanianetal05} is:
\begin{equation}
\label{TSSub}
T_{n}=\sup_{t\in [0,1]} \left|\,S_n(t)-t\,\right|\;.
\end{equation}
The motivation is best understood in the particular case where the
weights $w_i$ are constant. Then the function $g$ is also constant, and:
$$
S_{n}(t) = 
\sum_{k=1}^n \frac{1}{n}\,\mathbb{I}_{U_k\leqslant t}\;.
$$
This is
the empirical Cumulative Distribution Function (CDF) of the 
sample $(U_1,\ldots,U_n)$. The test statistic $T_{n}$ is the maximal
distance between that empirical CDF and the theoretical CDF of the
uniform distribution on the interval $[0,1]$. In other terms,
$\sqrt{n}T_{n}$ is the Kolmogorov Smirnov (KS) test statistic for
the goodness-of-fit of the uniform distribution on $[0,1]$ to 
the sample $(U_1,\ldots,U_n)$
(\cite{Arnold_Emerson_11}).
The constant weight case was initially proposed by 
\cite{Mootha03}, who explicitly referred to the KS
statistic (see also \cite[Supporting text, p.~5,6,11]{Subramanianetal05}, 
\cite{Ycartetal14}, and \cite{Tarca13}).
In the general case where the weights are not constant,
the distribution of the test statistic $T_n$ under the null hypothesis is
unknown. In the current implementations, it is approximated
by Monte-Carlo simulation on $1000$ random samples (\cite{Subramanianetal07}). 

Our first remark is that in the non constant case, the limit of
$S_n(t)$ as $n$ tends to infinity is not $t$, as (\ref{TSSub}) seems to
suggest, but instead:
$$
\lim_{n\to\infty} S_n(t)
=
\frac{\int_0^t g(u)\,\mathrm{d} u}
{\int_0^1 g(u)\,\mathrm{d} u}\;.
$$
Thus the GSEA test statistic $T_{n}$ is not appropriately
centered, unless the weights are constant. Instead, the following test
statistic should be used:
\begin{equation}
\label{TSWKS}
T^*_{n}=\sqrt{n}\sup_{t\in [0,1]} \left|\,S_n(t)-
\frac{\int_0^t g(u)\,\mathrm{d} u}
{\int_0^1 g(u)\,\mathrm{d} u}
\,\right|\;.
\end{equation}
The objective of this paper is to derive the asymptotic
distribution of $T^*_{n}$ under the null
hypothesis, then deduce from the mathematical 
result a practical testing procedure, 
and compare the outputs of that procedure to those of 
the classical GSEA test.

Our theoretical result is the following.
\begin{theo}
\label{th:main}
Let $g$ be a continuous, positive function from $[0,1]$ into
$\mathbb{R}$. Denote by $G$ its primitive: $G(t) =
\int_0^t g(u)\,\mathrm{d}u$, and assume that $G(1)=1$. Let
$(U_n)_{n\in\mathbb{N}}$ be a sequence of i.i.d. random variables, uniformly
distributed on $[0,1]$. For all $n\geqslant 1$, and for all 
$t$ in $[0,1]$, consider the random variable $S_n(t)$ defined by
\textrm{(\ref{ECSP})}. Let
\begin{equation}
\label{defZn}
Z_n(t) = \sqrt{n}\left(S_n(t) - G(t)\right)\;.
\end{equation}
As $n$ tends to infinity, the stochastic
process $\{Z_n(t)\,,\;t\in[0,1]\}$
converges weakly in $\ell^{\infty}([0,1])$  
to the process 
$\{Z(t)\,,\;t\in[0,1]\}$, where:
\begin{equation}
\label{defZ}
Z(t) =
\int_0^tg(u)\,\mathrm{d}W_u-G(t)\int_0^1g(u)\,
\mathrm{d}W_u\;, 
\end{equation}
and $\{W_t\,,\;t\in [0,1]\}$ is the standard Brownian motion.
\end{theo}
The hypothesis $\int_0^1g(u)\,\mathrm{d}u=1$ induces no loss of
generality: since $g$ is continuous and positive, its integral is
positive; $g$ can be divided by its integral without changing the
values of the cumulated proportion of weights $S_n(t)$. 
The proof of Theorem \ref{th:main}
will be given in section \ref{proof}. It is based on the
theory of empirical processes, for which 
\cite{ShorackWellner86} and 
\cite{Kosorok_08} will be used as 
general references. 

The first consequence of Theorem \ref{th:main} for GSEA, is that as
$n$ increases, the distribution of the proposed test statistic
$T^*_{n}$ under the null hypothesis, tends to that of the following
random variable $T$:
$$
T = \sup_{t\in [0,1]} |Z(t)|\;,
$$
where the random process $Z$ is defined by
(\ref{defZ}). Denote by $F$ its CDF: for all $x>0$,
\begin{equation}
\label{defcdf}
F(x) = \mbox{Prob}(T\leqslant x)
\;.
\end{equation}
Observe that $F(x)$ only depends on $g$, i.e. on the weights of the
vector to be tested. Except in the classical KS case of constant
weights, $F$ does not have a closed-form expression, but a Monte-Carlo
approximation is easily obtained. The testing procedure generalizes
that of the classical KS test:
since the test statistic $T^*_{n}$ has asymptotic CDF $F$ 
under the null hypothesis, the p-value of an observation
$T^*_{n}=x$ is $1-F(x)$. That testing procedure will
be referred to as \emph{Weighted Kolmorov Smirnov} (WKS) test. A crucial
feature is that, since $F$ only depends on the
weights, the same evaluation of $F$ can be repeatedly used for many
gene sets, which saves computing time. 
Of course, the repeated application of a test to a full database
of several thousand gene sets poses the problem of False Discovery
Rate (FDR) correction. In applications, we have used the
method  of \cite{Benjamini_Yekutieli01}: see \cite{DutoitvanderLaan07}
for multiple testing procedures in genomics.

Like the KS test, the WKS test is based on an asymptotic result. 
In practice, it is used for finite values of $n$. 
Therefore, it is necessary to determine for which size $n$ of
gene sets, the test can be applied with good precision. 
A Monte-Carlo comparison of the cumulative distribution function 
of $T^*_{n}$ to its limit $F$ for different values of $n$ was conducted. 
Our conclusion is that the test can be safely applied for 
gene set sizes $n$ larger than $40$.
Beyond Monte-Carlo validation, it was necessary to compare 
the outputs of the WKS test to those of the
classical GSEA test, on real data.
Inside the GEO dataset GSE36133 of \cite{Barretinaetal12}, we have
selected vectors (samples) from different types of tumors. 
These vectors were tested 
against all gene sets of MSig database C2, calculating for each sample the
p-values of both tests. The gene sets known to
be related to the same type of cancer as 
the initial vector were of particular interest. 
An example corresponding to a sample of liver tumor will be reported;
we consider it as typical of the observations that were made with
other samples.
The obtained results are encouraging: the WKS test tends to output
less significant gene sets than the classical GSEA test out of the whole
database, but more out of those gene sets related to the correct type
of cancer. Our conclusion is that, beyond its mathematical and
algorithmic advantages, the WKS test could be more informative in many
cases, than the classical GSEA test.

The document is organized in the following way. In section
\ref{proof}, Theorem \ref{th:main} is proved, and 
the asymptotic distribution of $T^*_{n}$ is deduced.
Section \ref{stats} is devoted to the statistical application,
beginning with the description of the Monte-Carlo algorithm of 
calculation of p-values. Results of simulated tests are
reported next. Finally, an example of 
comparison of the WKS test with the GSEA test on
real data is discussed.

\section{Theoretical background}
\label{proof}
The notations and results of
\cite{Kosorok_08} will be used. 
In particular, throughout the section, $\rightsquigarrow$ denotes
the weak convergence of processes in $\ell^{\infty}([0,1])$.
We first give the proof of Theorem \ref{th:main}, which asserts the
convergence 
$Z_n\rightsquigarrow Z$, where $Z_n$ is the empirical process defined
by (\ref{defZn}), and $Z$ is the Gaussian bridge defined by (\ref{defZ}).  
\begin{proof}
The idea is the following. Consider:
\begin{equation}
\label{Z1n}
Z^1_n(t) = \frac{\sum_{k=1}^n g(U_k)}{n}\,Z_n(t)\;. 
\end{equation}
Using the general results on empirical processes and Donsker classes,
 exposed in section 9.4
of \cite{Kosorok_08}, it will be proved that
$Z^1_n\rightsquigarrow Z$.
By the law of large
numbers, 
$$
\lim_{n\to\infty} \frac{\sum_{k=}^n g(U_k)}{n} = \int_0^1
g(u)\,\mathrm{d}u
=1\;,\quad \mbox{a.s.}
$$ 
The convergence  $Z_n\rightsquigarrow Z$
follows as 
an application of Slutsky's theorem: Theorem 7.15 of
\cite[p.~112]{Kosorok_08}. 

The random variable  $Z^1_n(t)$ can be written as follows:
\begin{eqnarray*}
Z^1_n(t)&=&
\frac{1}{\sqrt{n}}\left(\sum_{k=1}^ng(U_k)\mathbb{I}_{\{U_k \leqslant t\}}
-G(t)\sum_{k=1}^ng(U_k)\right)\\[2ex]
&=& 
\frac{1}{\sqrt{n}}
\left(\sum_{k=1}^ng(U_k)\left(\mathbb{I}_{\{U_k\leqslant t\}}
-G(t)\right)\right) \;,
\end{eqnarray*}
denoting by $G$ the primitive of $g$, as before.
Empirical processes are customarily written as function-indexed
processes. Define the class of functions $\mathcal{F}$ by:
$$
\mathcal{F}=
\left\{\,
g(\cdot)\left(\mathbb{I}_{[0,t]}(\cdot)-G(t)\right)
\,;\; t\in [0,1]\,\right\}\;.
$$
Denote by $\mathbb{P}_n$ the empirical measure of $(U_1,\ldots,U_n)$,
by $\mathbb{P}$ the uniform distribution on $[0,1]$, by $\mathbb{P}_n
f$ and 
$\mathbb{P}f $ the integrals of $f$ with respect to 
$\mathbb{P}_n$ and $\mathbb{P}$
\cite[p.~11]{Kosorok_08}. For $f\in
\mathcal{F}$, define $\tilde{Z}^1_n(f)$ by:
\begin{equation}
\tilde{Z}^1_n(f) = \sqrt{n}\left(\mathbb{P}_nf-Pf\right)\;.
\label{indexing_by_functions}
\end{equation}
Obviously, for all $t\in [0,1]$, 
\begin{equation}
\label{ZtildeZ}
Z^1_n(t)=\tilde{Z}^1_n\left(g(\cdot)\left(
\mathbb{I}_{[0,t]}(\cdot)-G(t)\right)\right) \;.
\end{equation}
Let us prove that $\mathcal{F}$ is a Donsker class.
Firstly, observe that the following class $\mathcal{F}_1$ is Donsker. 
$$
\mathcal{F}_1=\left\{\,\mathbb{I}_{[0,t]}(\cdot)-G(t)\,;\; t\in [0,1]\,\right\}\;.
$$
Indeed, for $f\in \mathcal{F}_1$, the process
$\sqrt{n}\left(\mathbb{P}_nf-\mathbb{P}f\right)$ converges
weakly to the standard Brownian bridge.
Since all functions in $\mathcal{F}_1$ take
values between $-1$ and $1$, the supremum of $|\mathbb{P}f|$ over 
$\mathcal{F}_1$ is not larger than $1$.
The function $g$, being continuous on a compact interval, is
bounded and measurable.
From Corollary 9.32, p.~173 of \cite{Kosorok_08}, it follows that
$\mathcal{F}$ is also Donsker.
The convergence of $\tilde{Z}^1_n$ now follows from the result of
\cite[p.~11]{Kosorok_08}. The limit $\tilde{Z}^1$ is a zero mean,
$\mathcal{F}$-indexed, Gaussian process. Its covariance function is
defined, for all $f_1, f_2$ in $\mathcal{F}$ by:
\begin{equation}
\label{covtildeZ}
\mathbb{E}[ \tilde{Z}^1(f_1)\,\tilde{Z}^1(f_2)]
=
\mathbb{P}(f_1f_2)-\mathbb{P}f_1\,\mathbb{P}f_2 \;.
\end{equation}
Through (\ref{ZtildeZ}), the convergence of $\tilde{Z}^1_n$ induces
the convergence of $Z^1_n$, to a zero mean, $[0,1]$-indexed
process $Z^1$. Let us compute the covariance function of $Z^1$.
For $s,t$ in $[0,1]$, let:
$$
f_1(\cdot) = 
g(\cdot)(\mathbb{I}_{[0,s]}(\cdot)-G(s))
\quad\mbox{and}\quad
f_2(\cdot) = 
g(\cdot)(\mathbb{I}_{[0,t]}(\cdot)-G(t))\;.
$$
Applying (\ref{covtildeZ}) to these functions $f_1$ and $f_2$ yields,
\begin{equation}
\label{covZ}
\begin{array}{lcl}
\mathbb{E}[ Z^1(t)\,Z^1(s)]
&=&\displaystyle{
\int_0^{\min(t,s)}g^2(u)\,\mathrm{d}u-
G(t)\,\int_0^sg^2(u)\,\mathrm{d}u} \\[2ex]
&&\displaystyle{
-G(s)\,\int_0^tg^2(u)\,\mathrm{d}u
+G(s)G(t)\,\int_0^1g^2(u)\,\mathrm{d}u\;.}
\end{array}
\end{equation}
There remains to be proved that $Z^1$ and $Z$ have the same
distribution, where $Z$ is defined by the representation
(\ref{defZ}) in terms of the standard Brownian motion $W$:
$$
Z(t) = \int_0^tg(u)\,\mathrm{d}W_u -
G(t)\,\int_0^1g(u)\,\mathrm{d}W_u \;.
$$
It is a well known fact that the primitive of a deterministic function
with respect  to the Brownian motion is Gaussian: therefore $Z$ is a
Gaussian process.
The covariance function is easily calculated, using formula (32), p.~128 
of \cite{ShorackWellner86}: it
is indeed defined by (\ref{covZ}). The processes $Z^1$ and
$Z$ are both Gaussian, their means and covariance are equal,
therefore they have the same distribution.
\end{proof}

As explained in the introduction, the random variable of interest for
GSEA is the supremum of the process $|Z|$ over the
interval $[0,1]$. 
\begin{coro}
\label{co:sup}
Under the notations and hypotheses of Theorem \ref{th:main},
let 
$$
T^*_n=\sup_{t\in [0,1]}\left|\,Z_n(t)\,\right|\;.
$$ 
Then $T^*_n$
converges in distribution to
$$
\sup_{t\in [0,1]} \left|Z(t)\right|
=\sup_{t\in [0,1]}
\left|\,
\int_0^tg(u)\mathrm{d}W_u-
\int_0^tg(u)\mathrm{d}u\int_0^1g(u)\mathrm{d}W_u\,\right|\;,
$$ 
where $W$ denotes the standard Brownian motion.
\end{coro}
\begin{proof}
The mapping $f\mapsto \sup_{t\in
  [0,1]}|f(t)|$, from $l^\infty([0,1])$ into $\mathbb{R}^+$,  
is continuous. From Theorem \ref{th:main},
$Z_n\rightsquigarrow Z$. The conclusion
follows as an application of Theorem 7.7, p.~109 of \cite{Kosorok_08}.
\end{proof}

\section{Statistical Application}
\label{stats}
\subsection{Implementation}
The R code (\cite{R_software}) implementing the WKS test has been made
available online, together with a user manual and samples of data.
Several issues regarding the implementation
are discussed here. The essential step is the evaluation of
 the cumulative distribution function distribution $F$ defined by
 (\ref{defcdf}), or else:
\begin{equation}
\label{lim:rv}
F(x)=\mbox{Prob}\left(
\sup_{t\in
  [0,1]}\left|\,\int_0^tg(u)\,\mathrm{d}W_u
-G(t)\,\int_0^1g(u)\,\mathrm{d}W_u\,\right|\leqslant
x\right) \;.
\end{equation}
A Monte-Carlo calculation has to be used. First of all,
sample paths for the stochastic process 
$$
\left\{\,\int_0^tg(u)\,\mathrm{d}W_u\,;\; t\in [0,1]\,\right\}
$$ 
must be simulated. This is done using a standard Euler-Maruyama
scheme: see
\cite{Sauer13} for a review of numerical methods for stochastic
integrals and differential equations. 
A regular subdivision of the interval $[0,1]$ into $m$
intervals is chosen:
$$
t_i = \frac{i}{m}\,,\; i=0,\ldots,m\;.
$$
Recall that in practice, the function $g$ is known at points
$i/N$ representing the genes. Hence it is natural to
choose $m=N$. 
The stochastic integral is approximated by a sum:
\begin{equation}
\int_0^tg(u)\,\mathrm{d}W_u
 \approx \sum_{i=0}^{m-1}g(t_i)\,(W_{t_{i+1}\wedge t}-W_{t_i\wedge t}\,)\;.
\label{stochastic_integral:approximation}
\end{equation}
The increments $W_{t_{i+1}}-W_{t_i}$ are easily simulated as i.i.d centered Gaussian
variables, with variance $1/m$.
An estimate of the CDF $F$ is obtained by simulating $nsim$ 
discretized trajectories of $Z$, taking the
maximum of the absolute value of each, then returning the empirical
CDF of the obtained sample. The algorithm can be written as follows.
   
\begin{algorithm}
\caption{Approximation of $F$}
\label{algo:limdistr}
\begin{algorithmic}[1]
\STATE Simulate increments of the Brownian motion on $t_0,\ldots,t_m$,
\STATE for $i=0,\ldots,m-1$, compute
$
g(t_i)\,(W_{t_{i+1}}-W_{t_i})
$,
\STATE get cumulated sums of the previous sequence,
\STATE deduce the discretized trajectory for 
$\{Z(t)\,,\; t\in [0,1]\}$ at $t_0,\ldots,t_m$,
\STATE compute the maximum absolute value of the previous sequence,
\STATE repeat $nsim$ times steps 1 to 5,
\STATE return the empirical distribution function of the obtained sample.
\end{algorithmic}
\end{algorithm}

Actually, since $F(x)$ is evaluated as the proportion of a sample
below $x$, the result must take the uncertainty into account. We
propose to return the lower bound of the 95\% left-sided
confidence interval,
instead of the point estimate. This gives an upper bound for the
p-value, which is a conservative evaluation. 
As stated before, the CDF $F$ only depends on the weight function
$g$. The relation between $g$ and $F$ is illustrated
on Figure \ref{fig:limdistrs}. Five different CDF's have been
computed, for  $g_k(x)=(k+1)(1-x^{1/k})$, $k=0,1,2,3,4$. Denote them
by $F_0,\ldots,F_4$. The case
$k=0$ is that of constant weights, and can be used as a validation for
the algorithm above: $F_0$ is the Kolmogorov Smirnov CDF, 
which has an explicit expression. It can be checked that the estimate 
output by Algorithm
\ref{algo:limdistr} is close to the known exact function. 
The curves of Figure \ref{fig:limdistrs} were obtained 
via $20\,000$ Monte-Carlo simulations,
over $15\,000$ discretization points.
It turns out
that for all $x$, $F_0(x)>\cdots>F_4(x)$: the steeper $g$, the smaller
$F$, and the larger the p-values. The differences between the curves
are sizable: calculating $\sup|F_k-F_0|$ for $k=1,\ldots,4$ gives  
$0.199$, $0.271$, $0.324$, $0.356$. 
\begin{figure}[!ht]
\centerline{
\includegraphics[width=9cm]{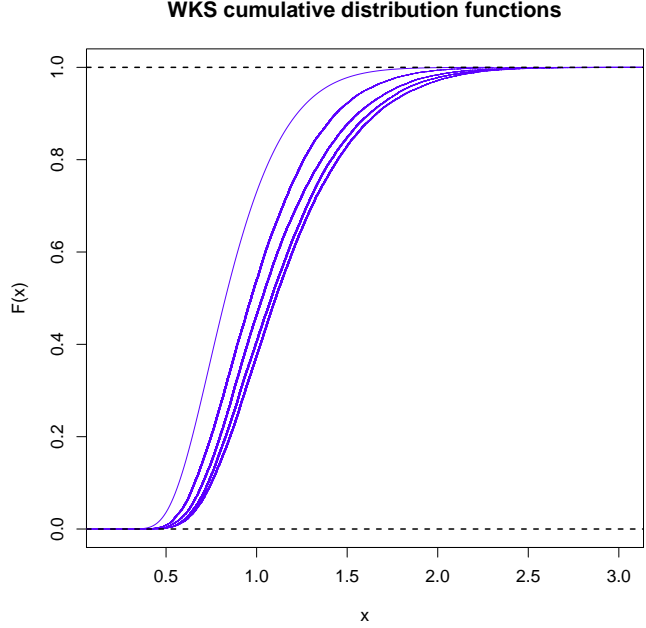}
}
\caption{Cumulated distribution functions $F_k$ corresponding to 
$g_k(x)=(k+1)(1-x^{1/k})$, for $k=0,1,2,3,4$. The highest curve 
corresponds to $k=0$ (constant weights, classical Kolmogorov Smirnov
CDF). The CDF's decrease as $k$ increases: the steeper $g$, the smaller
$F$, and the larger the p-values.
}
\label{fig:limdistrs}
\end{figure}
Theoretical functions $g$ may seem of little practical interest. This
is not so, for two reasons. The first reason is the use of robust
statistics (see \cite{Heritieretal09} as a general reference, and
\cite{Tsodikovetal02} for application to expression data). If the
initial values are replaced by their ranks, then the weights are
$N,N\!-\!1,\ldots,2,1$. Therefore, the weight function is
$g_1(x)=2(1-x)$.
This justifies calculating $F_1$ with good precision, which
makes the WKS test fast and precise, for all uses over
rank statistics. We
have done so, using $10^6$ Monte-Carlo simulations, and $10^5$
discretization points. The second reason is the observation of $F$
when the weights come from real data.
Eight different GEO datasets were considered:
GSE36382 (\cite{Mayerleetal13}), GSE48348 (\cite{EskoMetspalu13}), 
GSE36809 (\cite{Xiaoetal11}), 
GSE31312 (\cite{Freietal13}), GSE48762 (\cite{Obermoseretal13}), 
GSE37069 (\cite{Seoketal13}), GSE39582 (\cite{Marisaetal13}), and 
GSE9984 (\cite{Mikheevetal08}).
Several samples of expression levels in each study were selected. 
In each sample, the expression levels
were ranked in decreasing order, 
and Algorithm \ref{algo:limdistr} was applied
in order to obtain an estimation of $F$.
For all real datasets, the estimated $F$ was such that $F_4(x)<
F(x)<F_0(x)$. It seems to be the case in practice that $F_4$ and $F_0$ provide
lower and upper bounds for $F$.

The next algorithmic point concerns 
the calculation of the test
statistic, that is the value of $T^*_n$ defined by (\ref{TSWKS}) for a given 
set of weights and a gene set of size $n$:
$$
 T^*_{n}=\sqrt{n}\sup_{t\in [0,1]} \left|\,S_n(t)-
G(t)
\,\right|\;,
$$
where 
$$
S_n(t) = \frac{\sum_{k=1}^n g(U_k)\,\mathbb{I}_{U_k \leqslant t}}
{\sum_{k=1}^n g(U_k)}\;.
$$
The values $g(U_k)$ are the weights of genes inside the gene set.
Observe that, if the same vector has to be tested against many gene
sets, the calculation of $G(t)$ (cumulated sums of all weights) must be
done only once. 
The value of $T^*_n$ is returned by  
a procedure similar to that of the classical KS test. Consider two 
non-decreasing functions $f$ and $h$ where $f$ is a step function 
with jumps on the set $\{x_1,\ldots,x_n\}$ and $h$ is continuous.
The supremum of the difference between $f$ and $h$ is computed as follows 
\cite[p.~35]{Arnold_Emerson_11}.
$$
\sup_x|\,f(x)-h(x)\,|=
\max_i\{\,\max\{\,|\,h(x_i)-f(x_i)\,|,|\,h(x_i)-f(x_{i-1})\,|\,\}\,\}\;.
$$
\subsection{Validation of asymptotics on simulated data}
\label{stats:sim_data}
Since the WKS test relies on a convergence theorem, it is necessary to
determine the values of $n$ (the gene set size) for which the
procedure yields precise enough results. Such a validation is
standard. For a given $n$, a sample of gene sets of size $n$ is simulated, under
the null hypothesis. For each of them, the test statistic is computed,
thus a sample of values of the test statistic under the null
hypothesis is obtained. The
goodness-of-fit of the theoretical CDF $F$ to the empirical CDF of the
sample is tested by the (classical) KS test.
Figure \ref{fig:KSdist} shows results that
were obtained for
two functions $g$: one is $g_1(x)=2(1-x)$ (left panel), the other one
comes from real data: a sample in GSE36133  
of \cite{Barretinaetal12} (right panel). 
The evaluation of $F_1$ was done
over $10^6$ Monte-Carlo simulations, and $10^5$ discretization points,
as explained in the previous section. For the real data, the number of
discretization points was $m=N=18\,638$, and the number of
Monte-Carlo simulation was $nsim=20\,000$.
The values
of $n$ range from $5$ to $1\,100$ by step $5$. For each $n$,
$1\,500$ 
uniform random gene sets of size $n$ were simulated. The negative logarithm in
base $10$ of the KS p-value is plotted. On each plot the horizontal line
corresponding to a 5\% p-value has been added. The
p-values are small until $n=40$, they stay above 5\% after. This
is coherent with what is observed for most asymptotic tests, and
in particular the classical KS test. 
Beyond statistical validation, the comparison of the exact CDF,
estimated over random gene sets, with the theoretical asymptotic $F$
reveals an interesting feature of the WKS test: the exact CDF tends to
be smaller than $F$. This implies that the asymptotic p-value tends to
be larger than the true one, or else that the procedure is
conservative: small gene sets are less likely to be declared
significant by WKS.

\begin{figure}[!ht]
\centerline{
\begin{tabular}{cc}
\includegraphics[width=8cm]{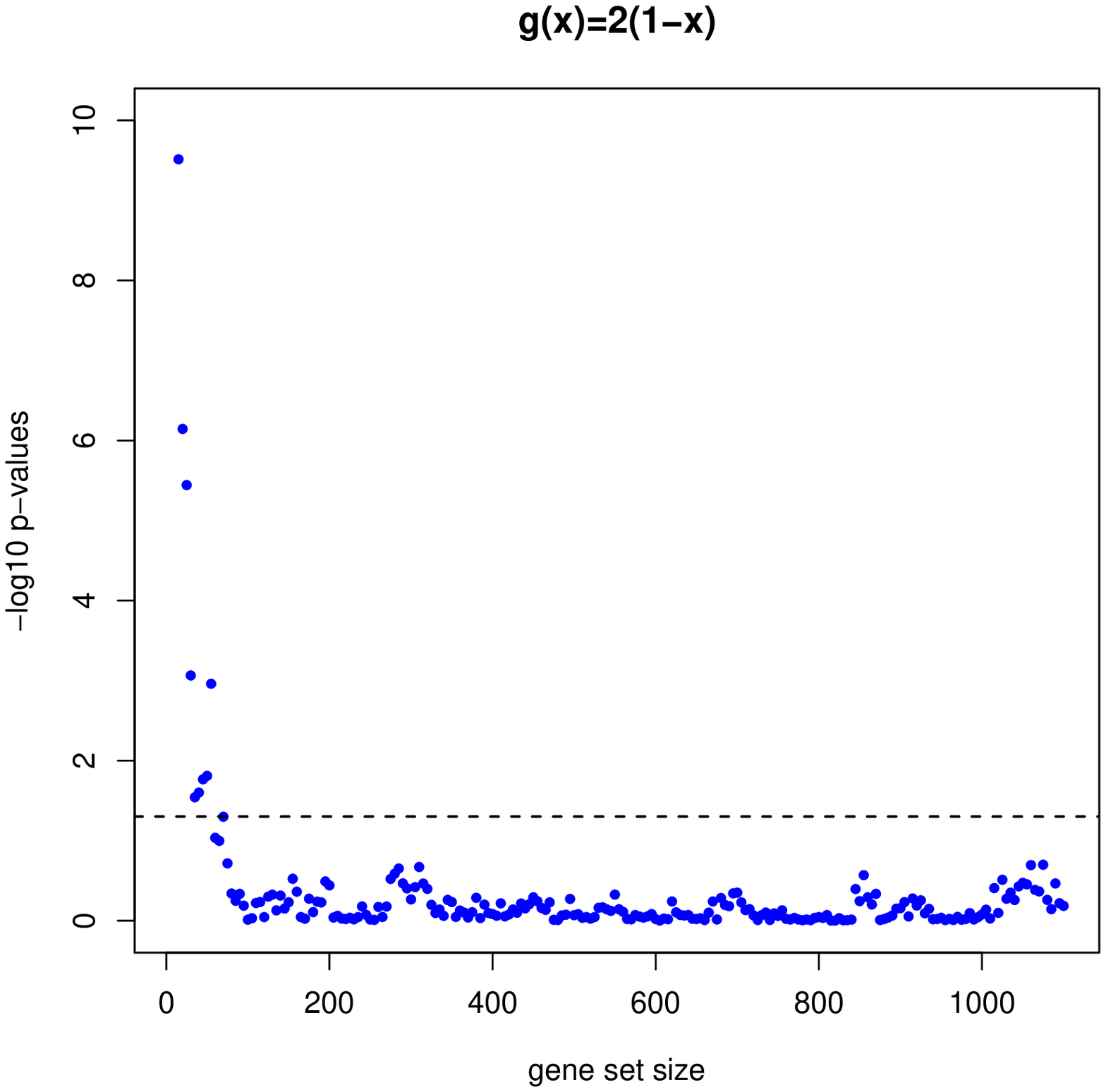}
&
\includegraphics[width=8cm]{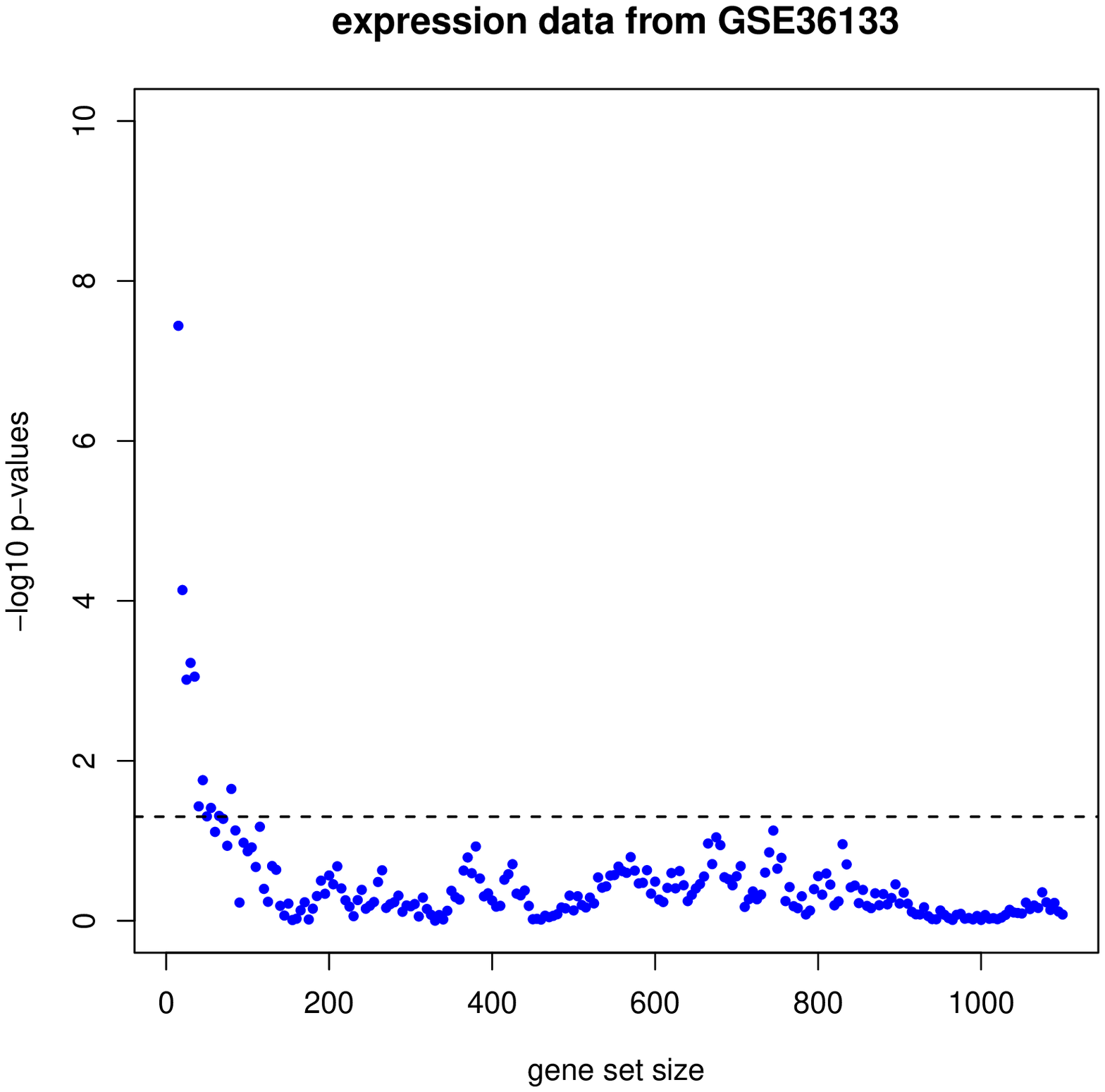}
\end{tabular}
}
\caption{Goodness-of-fit of simulated WKS test statistic $T^*_n$ over
  simulated gene sets. The function $g$ is $g(x)=2(1-x)$ on the left
  panel. It comes from real data on the right panel. The gene set size
  (abscissa) ranges from $5$ to $1\,100$ by step $5$. For each $n$ the
  ordinate is the negative logarithm in base 10 of the KS
  goodness-of-fit 
p-value, over a sample of $1\,500$ gene sets. The dashed
lines have ordinate $-\log_{10}(0.05)$.}
\label{fig:KSdist}
\end{figure}

On Figure \ref{fig:KSdist}, there is no clear difference 
between the theoretical $g$ (left), and real data
(right). However, it must be recalled that the null hypothesis H0,
under which simulations have been done in both cases, is 
that the gene set is a sample of uniform random variables on the
interval $[0,1]$. However, in practice, 
the gene set should be considered instead as a
random subset without replacement of the set of all genes. If the
gene set size $n$ is small compared to the total number of genes
$N$, the difference is negligible. We have conducted another set of
experiments, where gene sets were simulated 
by extracting random samples without replacement from 
$\{1/N,\ldots,N/N\}$. The results (not reported here), show a good
agreement with those of Figure  \ref{fig:KSdist}, until
$n=1\,000$. Beyond that value, the asymptotics becomes less precise.
It must be observed that
gene sets of size larger than $1\,000$ are relatively rare ($28$ out of
the $4\,722$ gene sets of C2).
\subsection{Comparison with classical GSEA}
\label{stats:real_data}
In this section, only real data are considered. Several vectors coming
from the GEO repository were tested against all $4\,722$
gene sets  in the MSig C2 database, using the classical GSEA, and the WKS
tests. The vectors that were used came from GEO dataset GSE36133 of
\cite{Barretinaetal12}, annotated using the org.Hs.eg.db package of
\cite{Carlson12}. This gave $N=18\,638$ different gene names. Observe that
applying the tests, the gene sets are necessarily reduced to those
$N$ genes. Out of the $21\,047$ different gene symbols
present in C2, only  $16\,683$ were common with the $N$ genes of the
chosen vectors.

For a given vector, two sets of $4\,722$ 
p-values were obtained, one with the GSEA test, the other with the WKS
test. Results that can be considered as typical are
represented on Figure \ref{fig:pvalues_GSEA_wKS}.
In that case, the vector contained expression data from liver
tumor tissue. Out of the $4\,722$ gene sets of C2,
$129$ have ``liver'' in their title. They were considered
are related to liver cancer, and the corresponding points are
represented as red triangles on the figure.
The negative 
logarithms in base 10 of the p-values of both tests have been plotted,
thus the figure displays $4\,722$ points corresponding to p-value pairs.
For comparison sake, only raw p-values are considered,
without FDR adjustment.
A p-value of 5\% is marked by a dashed black line: points on
the right of the vertical line are significant for the
classical GSEA test, points above the horizontal line
are significant for the WKS test.
For the WKS test,
the CDF $F$ was calculated over 
$m=N=18\,638$ discretization points, and the number of
Monte-Carlo simulations was $nsim=10^5$.
For the classical GSEA test, the number of
Monte-Carlo simulation had to be limited to $10^4$. 
\begin{figure}[!ht]
\centerline{
\includegraphics[width=9cm]{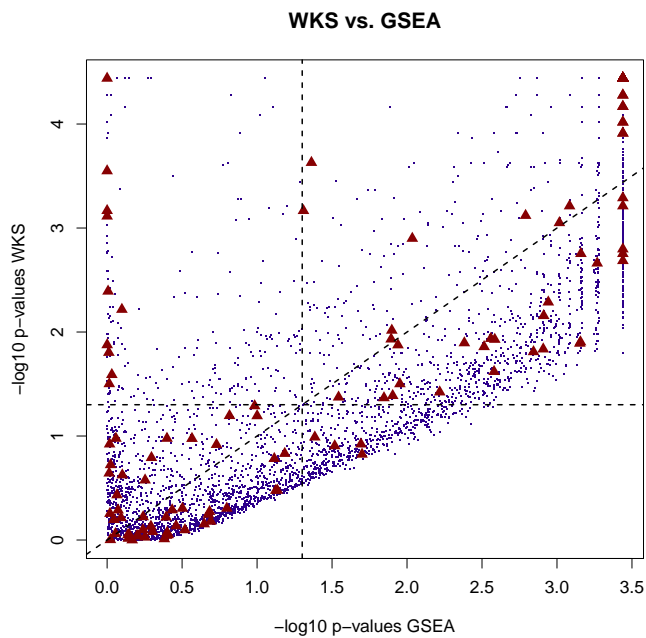}
}
\caption{Test of a liver tumor expression vector against the $4\,722$
  gene sets of the MSig C2 database. Each point corresponds to a gene
  set, the coordinates being the negative logarithm in base 10 of the p-values,
  for the classical GSEA and the WKS tests. Gene sets related to liver
  cancer in the database are represented as red triangles. The
  horizontal and vertical dashed lines
  correspond to 5\% p-values.}
\label{fig:pvalues_GSEA_wKS}
\end{figure}

The vertical dotted lines appearing on the right of the graphic are artefacts,
due to the Monte-Carlo method for the GSEA test: the rightmost line
corresponds to cases where the point-estimated p-value is equal to
$0$. Apart from these artefacts, it must be observed that the
results of both tests are globally coherent: 
 $2\,501$ database gene sets were significant (p-value smaller than
 5\%) for the  WKS test, $2\,764$ for GSEA, $2\,268$ for
 both.  There are no points in
 the bottom right corner of the graphics: when a p-value is very small
 for GSEA, it is never large for WKS. The converse is not true: many
 points in the upper left corner correspond to gene sets with a large
 p-value for GSEA, small for WKS.

More interesting is the analysis of
 liver-related gene sets. Out of $129$, $76$ were declared significant
 by the WKS test; $70$ by the GSEA test, $66$ by both. 
Therefore, $10$ gene sets were
declared significant by WKS only, and $4$ by GSEA only. 
Figure \ref{fig:pathways_different_for_two_tests} plots the cumulated
proportions of weights
$S_n(t)$ for those $14$ gene sets. On the same plot, the functions $t$
(bisector), to which the classical GSEA test compares $S_n(t)$, and
$G(t)$, used as a centering by WKS, also appear. On the
graphic, the reason why a gene set may be declared significant by one
test and not the other, is clear. The $4$ gene sets declared
significant by GSEA and not WKS, are represented by blue step
functions; they are above the $G$ curve. They are 
indeed far from the bisector, but not far enough from
$G$. Inside the corresponding gene sets, the weights of the
genes tend to be representative of the global distribution of weights,
and declaring them as significant by comparing to the bisector can
be regarded as a bias. Moreover, it should be observed that
$3$ out of the $4$ have size below $19$.
As already explained, when dealing with very small sizes, the WKS test tends
to underrate significance.

Conversely,
the 10 gene sets declared significant by WKS and not GSEA are
represented by red step functions.
They are relatively close to the bisector as expected, but clearly
below the $G$ curve, to which WKS compares. This means that in the
corresponding gene sets, the genes tend to have significantly smaller
weights, i.e. they are significantly underexpressed. An interesting
example is the gene set named
\verb+Acevedo_methylated_in_liver_cancer_dn+. As indicated by the two
letters \verb+dn+, it contains genes which are known to be
down-regulated in case of liver cancer (\cite{Acevedo08}). 
On Figure \ref{fig:pvalues_GSEA_wKS}, it appears
on the upper left corner: it has p-value close to $0$ for WKS, close to
$1$ for GSEA. Thus WKS has detected it as significantly related to the
tested vector, whereas GSEA has not. The case is not unique: 3 gene sets
had p-value larger than $0.5$ for GSEA, smaller than $10^{-3}$ for WKS.

\begin{figure}[!ht]
\centerline{
\includegraphics[width=10cm]{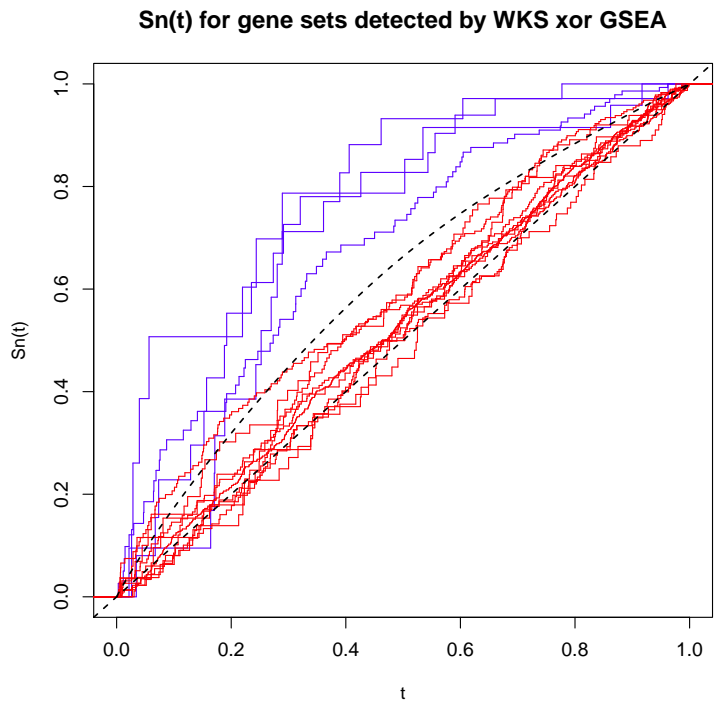}
} 
\caption{Plots of the cumulated weight function 
$S_n(t)$ for vectors declared
significant by WKS and not GSEA (red step functions) and
conversely (blue step functions). The functions $t$
(to which the classical GSEA test compares $S_n(t)$), and
$G(t)$ (used as a centering by WKS), are dashed.}
\label{fig:pathways_different_for_two_tests}
\end{figure}

As already stated, these results were consistently observed for
different expression vectors, from different types of cancers.
In all cases, WKS declared less significant pathways than GSEA 
in a proportion of about $10\%$ from the whole database, whereas it
tended to detect more significant gene sets among those related to the
correct type of cancer.
\section{Conclusion}
A new method for testing the relative enrichment of a gene set,
compared to a vector of numeric data over the whole genome, has been
proposed. Like   
the classical GSEA test of \cite{Subramanianetal05}, it is based on
cumulated proportions of weights, but a different centering is used. 
A convergence result that generalizes the
classical Kolmogorov Smirnov theorem, has been obtained. 
The corresponding testing
procedure extends the standard Kolmogorov Smirnov test and has been
called Weighted Kolmogorov Smirnov (WKS).
A major advantage of the WKS
test is that the calculation of p-values only depends on the
vector to be tested, and not on the gene set. Therefore, the same
distribution function can be used for calculating p-values over many
gene sets. A
Monte-Carlo evaluation has shown that the procedure is precise for
values of the gene set size larger than $40$. For a set of less
than $40$ genes, the WKS test is conservative, in the sense that
the p-value is increased, and therefore
the gene set is less likely to be declared significant. For
statistical coherence, the gene set size should not be
larger than $1\,000$. The WKS test has been compared with the classical
GSEA test over expression vectors of tumors coming from the GEO dataset
GSE36133 of \cite{Barretinaetal12}, tested against the MSig database
C2 (\cite{Subramanianetal05}). The comparison has shown that the results
of both tests are globally coherent. The WKS test tends to output less
significant gene sets out of the whole database, but more out of gene
sets specifically related to the same type of tumor. In particular,
the WKS test detects sets of underexpressed genes
which are not significant for GSEA. This encouraging
result needs to be consolidated, by using the WKS test over 
different types of vectors, and
more databases of gene sets.

Like the GSEA test, the WKS test can be used on any type of numeric
data. In particular, a transformation can be applied to the raw
expression levels before testing. In particular, the initial data can be
replaced by their ranks, in which case the test has low
computing cost, for a good precision. If, over the same database, 
the p-values of the initial vector, and the vector
of ranks are compared, a good agreement is
observed; yet less gene sets are declared significant against the
rank vector. Here we have considered only the two sided version of the
test: gene sets are declared significant when their cumulated
proportion of weights $S_n(t)$ is too far from the theoretical value
$G(t)$. Just like the KS test, the WKS can be made one-sided, by
testing the signed difference between $S_n(t)$ and $G(t)$: a gene set
for which $\inf (S_n(t)-G(t))$ is significantly negative, contains
genes whose weights tend to be small (down-regulated). Conversely,
gene sets for which $\sup (S_n(t)-G(t))$ is significantly positive,
contain more up-regulated genes.

Both the GSEA and the WKS tests have been implemented in a R
script. It is available online, together with data samples, and a
user manual, from the following address.\\
\centerline{\texttt{http://ljk.imag.fr/membres/Bernard.Ycart/publis/wks.tgz}}\\
We hope this will encourage further testing of the tool, and
validation in new biological studies.


\section*{Acknowledgements}
The authors acknowledge financial support from 
Laboratoire d'Excellence TOUCAN (Toulouse Cancer). They are indebted
to Alain Le Breton and Marina Kleptsyna for helpful remarks.
\end{bmcformat}

\begin{thebibliography}{18}
\newcommand{\enquote}[1]{``#1''}
\providecommand{\natexlab}[1]{#1}
\providecommand{\url}[1]{\texttt{#1}}
\providecommand{\urlprefix}{URL }


\bibitem[Acevedo et~al.(2008)]{Acevedo08}
Acevedo L.~G., Bieda M., Green R., and Farnham P.~J. (2008):
\enquote{Analysis
  of the mechanisms mediating tumor-specific changes in gene expression in
  human liver tumors,} \emph{Cancer Res.}, 68, 2641--51.

\bibitem[Arnold and Emerson(2011)]{Arnold_Emerson_11}
Arnold, T.~B. and Emerson J.~W. (2011):
\enquote{Nonparametric Goodness-of-Fit
Tests for Discrete Null Distributions,} \emph{The R Journal}, 3/2, 34--39.

\bibitem[Barretina et~al.(2012)]{Barretinaetal12}
Barretina J., Caponigro G., Stransky N., Venkatesan K., and others (2012):
\enquote{The Cancer Cell Line Encyclopedia enables predictive modelling of
  anticancer drug sensitivity,} \emph{Nature}, 483(7391), 603--7.

\bibitem[{Benjamini and Yekutieli(2001)}]{Benjamini_Yekutieli01}
Benjamini Y. and Yekutieli D. (2001): \enquote{The control of the false
  discovery rate in multiple testing under dependency,} \emph{Ann. Statist.},
  29, 1165--1188.

\bibitem[Bild and Febbo(2005)]{Bild05}
Bild A. and Febbo P.~G. (2005):
\enquote{Application of a priori established
  gene sets to discover biologically important differential expression in
  microarray data,} \emph{PNAS}, 102(43), 15278--15279.

\bibitem[Carlson(2012)]{Carlson12}
Carlson M. (2012):
\enquote{org.Hs.eg.db: Genome wide annotation for Human,} 
\emph{R package version 2.8.0}.

\bibitem[Dutoit and van der Laan (2007)]{DutoitvanderLaan07}
Dutoit, S. and van der Laan M.,
\emph{Multiple testing procedures with applications to genomics,}
Springer, New York, 2007.

\bibitem[Edgar et~al.(2002)]{Edgar02}
Edgar R., Domrachev M., and Lash A.~E. (2002):
\enquote{Gene Expression Omnibus: NCBI gene expression and
  hybridization array data repository,} 
  \emph{Nucleic Acids Res.}, 30, 207--210.


\bibitem[Esko and Metspalu(2013)]{EskoMetspalu13}
Esko T. and Metspalu A. (NCBI2013:Series GSE48348):
\enquote{Gene Expression
  profiling in healthy population samples,} \emph{Gene Expression Omnibus
  (GEO)}.

\bibitem[Frei et~al.(2013)]{Freietal13}
Frei E., Visco C., Xu-Monette Z.~Y., Dirnhofer S., and others (2013):
\enquote{Addition of rituximab to chemotherapy overcomes the negative
  prognostic impact of cyclin E expression in diffuse large B-cell lymphoma,}
  \emph{J Clin Pathol}, 66(11), 956--61.

\bibitem[H\'eritier et~al.(2009)]{Heritieretal09}
H\'eritier S., Cantoni E., Copt S., and {Victoria-Feser} M. P. (2009): 
\emph{Robust methods in biostatistics}, Wiley, New York.

\bibitem[Huang et~al.(2009)]{Huangetal09}
Huang D.~W, Sherman B.~T., and Lempicki R.~A. (2009):
\enquote{Bioinformatics
  enrichment tools: paths toward the comprehensive functional analysis of large
  gene lists,} \emph{Nucleic Acids Res.}, 37(1), 1--13.


\bibitem[Kosorok(2008)]{Kosorok_08}
Kosorok M.~R. (2008): 
\emph{Introduction to Empirical Processes 
and Semiparametric Inference,} Springer, New York.

\bibitem[Marisa et~al.(2013)]{Marisaetal13}
Marisa L., de Reyni\`es A., Duval A., Selves J., and others (2013):
\enquote{Gene
  expression classification of colon cancer into molecular subtypes:
  characterization, validation, and prognostic value,} \emph{PLoS Med}, 10(5),
  e1001453.

\bibitem[Mayerle et~al.(2013)]{Mayerleetal13}
Mayerle J., den Hoed C.~M., Schurmann C., Stolk L., and others (2013):
\enquote{Identification of genetic loci associated with Helicobacter pylori
  serologic status,} \emph{JAMA}, 309(18), 1912--20.

\bibitem[Mikheev et~al.(2008)]{Mikheevetal08}
Mikheev A.~M., Nabekura T., Kaddoumi A., Bammler T.~K., and others (2008):
\enquote{Profiling gene expression in human placentae of different
  gestational ages: an OPRU network and UW SCOR study,} \emph{Reprod Sci},
  15(9), 866--77.

\bibitem[Mootha et~al.(2003)]{Mootha03}
Mootha V.~K., Lindgren C.~M., Eriksson K.~F., Subramanian A., 
Sihag S., Lehar J., Puigserver  P., Carlsson E., Ridderstr{\aa}le M., 
Laurila E., and others (2003):
\enquote{PGC-1alpha-responsive genes involved in oxidative
  phosphorylation are coordinately downregulated in human diabetes,} \emph{Nat.
  Genet.}, 34, 267--273.

\bibitem[Nam and Kim(2008)]{Nametal08}
Nam D. and Kim S.~Y. (2008):
\enquote{Gene-set approach for expression pattern
  analysis,} \emph{Brief Bioinform}, 9(3), 189--197.

\bibitem[Obermoser et~al.(2013)]{Obermoseretal13}
Obermoser G., Presnell S., Domico K., Xu H. and others (2013):
\enquote{Systems
  scale interactive exploration reveals quantitative and qualitative
  differences in response to influenza and pneumococcal vaccines,}
  \emph{Immunity}, 38(4), 831--44.

\bibitem[{R Core Team}(2013)]{R_software}
{R Core Team} (2013):
\emph{R: A Language and Environment for Statistical
  Computing}, R Foundation for Statistical Computing, Vienna, Austria,
  \urlprefix\url{http://www.R-project.org/}, {ISBN} 3-900051-07-0.

\bibitem[Sauer(2013)]{Sauer13}
Sauer, T. (2013):
\enquote{Computational solution of 
stochastic differential equations,} \emph{WIREs Comput Stat} 2013.
doi: 101002/wics.1272.

\bibitem[Seok et~al.(2013)]{Seoketal13}
Seok J., Warren H.~S., Cuenca A.~G., Mindrinos M.~N., and others (2013):
\enquote{Genomic responses in mouse models poorly mimic human inflammatory
  diseases,} \emph{PNAS}, 110(9), 3507--12.

\bibitem[Shorack and Wellner(1986)]{ShorackWellner86}
Shorack G.~R. and Wellner J.~A. (1986):
\emph{Empirical Processes with
  Applications to Statistics}, Wiley, New York.

\bibitem[Subramanian et~al.(2005)]{Subramanianetal05}
Subramanian A., Tamayo P., Mootha V.~K., Mukherjee S., Ebert B.~L.,
Gillette M.~A., Paulovich A., Pomeroy S.~L., Golub T.~R., Lander E.~S. 
and Mesirov J.~P. (2005):
\enquote{Gene set enrichment analysis: A knowledge-based
  approach for interpreting genome-wide expression profiles,} \emph{PNAS}, 102,
  15545--50, \urlprefix\url{http://www.pnas.org/content/102/43/15545.full}.

\bibitem[Subramanian et~al.(2007)]{Subramanianetal07}
Subramanian A., Kuehn H., Gould J., Tamayo P., and Mesirov J.~P. (2007):
\enquote{Gsea-P: a desktop application for Gene Set Enrichment Analysis,}
  \emph{Bioinformatics}, 23(23), 3251--3.

\bibitem[Tarca et~al.(2013)]{Tarca13}
Tarca A.~L., Bhatti G., and Romero R. (2013):
\enquote{A Comparison of Gene 
  Set Analysis Methods in Terms of Sensitivity, Prioritization and 
  Specificity,} \emph{PloS one}, 8(11), e79217.

\bibitem[Tsodikov et~al.(2002)]{Tsodikovetal02}
Tsodikov A., Szabo, A., and Jones, D. (2002):
\enquote{Adjustments and measures of differential expression for microarray
	data,} \emph{Bioinformatics}, 18, 251--260.

\bibitem[Xiao et~al.(2011)]{Xiaoetal11}
Xiao W., Mindrinos M.~N., Seok J., Cuschieri J., and others (2011):
\enquote{A
  genomic storm in critically injured humans,} \emph{J Exp Med}, 208(13),
  2581--90.

\bibitem[Ycart et~al.(2014)]{Ycartetal14}
Ycart B., Pont F., and Fourni\'e J.~J. (2014):
\enquote{Curbing false
  discovery rates in interpretation of genome-wide expression profiles,}
  \emph{J Biomed Inform.}, 47, 58--61.

\end{thebibliography}
\end{document}